\documentclass[11pt,leqno]{article}
\usepackage{graphicx, amsfonts, amsthm, amsxtra, amssymb, verbatim, makeidx}
\usepackage{subeqnarray, relsize}
\usepackage[mathscr]{euscript}
\usepackage{hyperref}
\makeindex
\makeglossary
\begin{document}
\newtheorem{theo}{Theorem}
\newtheorem{exam}{Example}
\newtheorem{coro}{Corollary}
\newtheorem{defi}{Definition}
\newtheorem{prob}{Problem}
\newtheorem{lemm}{Lemma}
\newtheorem{prop}{Proposition}
\newtheorem{rem}{Remark}
\newtheorem{conj}{Conjecture}
\newtheorem{calc}{}


\def\gru{\mu} 
\def\pg{{ \sf S}}               
\def\TS{\mathlarger{\bf T}}                
\def\NB{{\mathlarger{\bf N}}}
\def\group{{\sf G}}
\def\NLL{{\rm NL}}   

\def\plc{{ Z_\infty}}    
\def\pola{{u}}      
\newcommand\licy[1]{{\mathbb P}^{#1}} 
\newcommand\aoc[1]{Z^{#1}}     
\def\HL{{\rm Ho}}     
\def\NLL{{\rm NL}}   

\def\Z{\mathbb{Z}}                   
\def\Q{\mathbb{Q}}                   
\def\C{\mathbb{C}}                   
\def\N{\mathbb{N}}                   
\def\uhp{{\mathbb H}}                
\def\A{\mathbb{A}}                   
\def\dR{{\rm dR}}                    
\def\F{{\cal F}}                     
\def\Sp{{\rm Sp}}                    
\def\Gm{\mathbb{G}_m}                 
\def\Ga{\mathbb{G}_a}                 
\def\Tr{{\rm Tr}}                      
\def\tr{{{\mathsf t}{\mathsf r}}}                 
\def\spec{{\rm Spec}}            
\def\ker{{\rm ker}}              
\def\GL{{\rm GL}}                
\def\ker{{\rm ker}}              
\def\coker{{\rm coker}}          
\def\im{{\rm Im}}               
\def\coim{{\rm Coim}}            
\def\p{{\sf  p}}
\def\U{{\cal U}}   

\def\weig{{\nu}}
\def\r{{ r}}                       
\def\k{{\sf k}}                     
\def\ring{{\sf R}}                   
\def\X{{\sf X}}                      
\def\Ua{{   L}}                      
\def\T{{\sf T}}                      
\def\asone{{\sf A}}                  

\def\Ts{{\sf S}}
\def\cmv{{\sf M}}                    
\def\BG{{\sf G}}                       
\def\podu{{\sf pd}}                   
\def\ped{{\sf U}}                    
\def\per{{\bf  P}}                   
\def\gm{{  A}}                    
\def\gma{{\sf  B}}                   
\def\ben{{\sf b}}                    

\def\Rav{{\mathfrak M }}                     
\def\Ram{{\mathfrak C}}                     
\def\Rap{{\mathfrak G}}                     

\def\mov{{\sf  m}}                    
\def\Yuk{{\sf C}}                     
\def\Ra{{\sf R}}                      
\def\hn{{ h}}                         
\def\cpe{{\sf C}}                     
\def\g{{\sf g}}                       
\def\t{{\sf t}}                       
\def\pedo{{\sf  \Pi}}                  

\def\Der{{\rm Der}}                   
\def\MMF{{\sf MF}}                    
\def\codim{{\rm codim}}                
\def\dim{{\rm    dim}}                
\def\Lie{{\rm Lie}}                   
\def\gg{{\mathfrak g}}                

\def\u{{\sf u}}                       

\def\imh{{  \Psi}}                 
\def\imc{{  \Phi }}                  
\def\stab{{\rm Stab }}               
\def\Vec{{\rm Vec}}                 
\def\prim{{\rm  0}}                  

\def\Fg{{\sf F}}     
\def\hol{{\rm hol}}  
\def\non{{\rm non}}  
\def\alg{{\rm alg}}  
\def\tra{{\rm tra}}  

\def\bcov{{\rm \O_\T}}       

\def\leaves{{\cal L}}        

\def\cat{{\cal A}}              
\def\im{{\rm Im}}               

\def\pn{{\sf p}}              
\def\Pic{{\rm Pic}}           
\def\free{{\rm free}}         
\def \NS{{\rm NS}}    
\def\tor{{\rm tor}}
\def\codmod{{\xi}}    

\def\GM{{\rm GM}}

\def\perr{{\sf q}}        
\def\perdo{{\cal K}}   
\def\sfl{{\mathrm F}} 
\def\sp{{\mathbb S}}  

\newcommand\diff[1]{\frac{d #1}{dz}} 
\def\End{{\rm End}}              

\def\sing{{\rm Sing}}            
\def\cha{{\rm char}}             
\def\Gal{{\rm Gal}}              
\def\jacob{{\rm jacob}}          
\def\tjurina{{\rm tjurina}}      
\newcommand\Pn[1]{\mathbb{P}^{#1}}   
\def\P{\mathbb{P}}
\def\Ff{\mathbb{F}}                  

\def\O{{\cal O}}                     

\def\ring{{\mathsf R}}                         
\def\R{\mathbb{R}}                   

\newcommand\ep[1]{e^{\frac{2\pi i}{#1}}}
\newcommand\HH[2]{H^{#2}(#1)}        
\def\Mat{{\rm Mat}}              
\newcommand{\mat}[4]{
     \begin{pmatrix}
            #1 & #2 \\
            #3 & #4
       \end{pmatrix}
    }                                
\newcommand{\matt}[2]{
     \begin{pmatrix}                 
            #1   \\
            #2
       \end{pmatrix}
    }
\def\cl{{\rm cl}}                

\def\hc{{\mathsf H}}                 
\def\Hb{{\cal H}}                    
\def\pese{{\sf P}}                  

\def\PP{\tilde{\cal P}}              
\def\K{{\mathbb K}}                  

\def\M{{\cal M}}
\def\RR{{\cal R}}
\newcommand\Hi[1]{\mathbb{P}^{#1}_\infty}
\def\pt{\mathbb{C}[t]}               
\def\gr{{\rm Gr}}                
\def\Im{{\rm Im}}                
\def\Re{{\rm Re}}                
\def\depth{{\rm depth}}
\newcommand\SL[2]{{\rm SL}(#1, #2)}    
\newcommand\PSL[2]{{\rm PSL}(#1, #2)}  
\def\Resi{{\rm Resi}}              

\def\L{{\cal L}}                     
\def\Aut{{\rm Aut}}              
\def\any{R}                          
\newcommand\ovl[1]{\overline{#1}}    

\newcommand\mf[2]{{M}^{#1}_{#2}}     
\newcommand\mfn[2]{{\tilde M}^{#1}_{#2}}     

\newcommand\bn[2]{\binom{#1}{#2}}    
\def\ja{{\rm j}}                 
\def\Sc{\mathsf{S}}                  
\newcommand\es[1]{g_{#1}}            
\newcommand\V{{\mathsf V}}           
\newcommand\WW{{\mathsf W}}          
\newcommand\Ss{{\cal O}}             
\def\rank{{\rm rank}}                
\def\Dif{{\cal D}}                   
\def\gcd{{\rm gcd}}                  
\def\zedi{{\rm ZD}}                  
\def\BM{{\mathsf H}}                 
\def\plf{{\sf pl}}                             
\def\sgn{{\rm sgn}}                      
\def\diag{{\rm diag}}                   
\def\hodge{{\rm Hodge}}
\def\HF{{ F}}                                
\def\WF{{ W}}                               
\def\HV{{\sf HV}}                                
\def\pol{{\rm pole}}                               
\def\bafi{{\sf r}}
\def\id{{\rm id}}                               
\def\gms{{\sf M}}                           
\def\Iso{{\rm Iso}}                           

\def\hl{{\rm L}}    
\def\imF{{\rm F}}
\def\imG{{\rm G}}

\begin{center}
{\LARGE\bf  Periods of linear algebraic cycles
}
\\
\vspace{.25in} {\large {\sc Hossein Movasati, Roberto  Villaflor Loyola}}\footnote{
Instituto de Matem\'atica Pura e Aplicada, IMPA, Estrada Dona Castorina, 110, 22460-320, Rio de Janeiro, RJ, Brazil,
{\tt www.impa.br/$\sim$ hossein, hossein@impa.br, rvilla@impa.br}}
\end{center}


\def\pn{{\sf p}}
\def\rootsG{{\sf G}}
\def\NLL{{\rm NL}}   

\begin{abstract}
In this article we use a theorem of Carlson and Griffiths and 
compute periods of linear algebraic cycles $\Pn{\frac{n}{2}}$ inside the Fermat variety of even 
dimension $n$ and degree $d$. As an application, for examples  of $n$ and $d$,  
we prove that the locus of hypersurfaces containing two linear cycles whose intersection is of low dimension, 
is a reduced  component of the Hodge locus in the underlying 
parameter space. We also check the same statement for hypersurfaces containing a complete intersection algebraic cycle.  
Our result confirms the Hodge conjecture for Hodge cycles
obtained by the monodromy of the homology class of such algebraic cycles. This is known as the variational Hodge conjecture. 
\end{abstract}
\section{Introduction}

\def\group{\mu}                 
\def\pg{{ \sf S}}               

Let us consider the even dimensional Fermat variety
\begin{equation}
\label{27nov2015}
X^d_n\subset \Pn {n+1}:   \ \  x_0^{d}+x_1^{d}+\cdots+x_{n+1}^d=0. 
\end{equation}
It has the following  linear algebraic cycles of dimension $\frac{n}{2}$:
\begin{equation}
\label{22jan2019}
 \licy{\frac{n}{2}}_{a,b}:  
\left\{
 \begin{array}{l}
 x_{b_0}-\zeta_{2d}^{1+2a_1}x_{b_1}=0,\\
 x_{b_2}-\zeta_{2d}^{1+2a_3} x_{b_3}=0,\\
 x_{b_4}-\zeta_{2d}^{1+2a_5} x_{b_5}=0,\\
 \cdots \\
 x_{b_n}-\zeta_{2d}^{1+2a_{n+1}} x_{b_{n+1}}=0,
 \end{array}
 \right. 
\end{equation}
where $\zeta_{2d}$ is a $2d$-primitive root of 
unity, $b$ is a permutation of $\{0,1,2,\ldots,n+1\}$ and  $0\leq a_i\leq d-1$ are integers.
In order to get distinct cycles we may further assume that $b_0=0$ and for $i$ an even number $b_i$ is the smallest number in
$\{0,1,\ldots,n+1\}\backslash\{b_0,b_1,b_2,\ldots, b_{i-1}\}$. 
It is easy
to see that the number of such cycles is $1\cdot 3\cdots (n-1)(n+1) d^{\frac{n}{2}+1}$ (for $d=3, n=2$ this is the famous $27$ lines in a smooth cubic surface). 
In this article we use a theorem of Carlson and Griffiths in \cite{CarlsonGriffiths1980} and prove the following: 
\begin{theo}
\label{14out2016}
For non-negative integers $i_0,i_1,\cdots, i_{n+1}$  with $\sum_{k=0}^{n+1}i_k=(\frac{n}{2}+1)d-n-2$, we have
\begin{equation}
\label{21march2017}
\frac{1}{(2\pi \sqrt{-1})^{\frac{n}{2}}}\mathlarger{\mathlarger{\int}}_{\licy{\frac{n}{2}}_{a,b}}
{\rm Residue}\left(  
\frac{x_0^{i_0}x_1^{i_1}\cdots x_{n+1}^{i_{n+1}}\cdot  \sum_{i=0}^{n+1}(-1)^ix_i\widehat{dx_i}}
     {( x_0^{d}+x_1^{d}+\cdots+x_{n+1}^d)^{\frac{n}{2}+1}}\right)
= 
\end{equation}
$$
\left\{
	\begin{array}{ll}
		   \frac{{\rm sign}(b)}{d^{\frac{n}{2}+1}\cdot \frac{n}{2}!} 
		   {\zeta_{2d}}^{ \sum_{e=0}^\frac{n}{2} (i_{b_{2e}}+1)\cdot (1+2a_{2e+1})}  & \mbox{if } \ \ \  
		   \ i_{b_{2e-2}}+i_{b_{2e-1}}=d-2, \ \ \ \forall e=1,...,\frac{n}{2}+1, \\
		0 & \mbox{otherwise. } 
	\end{array}
\right.
$$
where $\zeta_{2d}=e^{\frac{\pi i}{d}}$ is the $2d$-th primitive root of unity.
\end{theo}
For the residue map see \S\ref{6oct2017}. 
Using Theorem \ref{14out2016} we can prove a stronger version of the variational Hodge conjecture for many  algebraic cycles, see \cite[page 103]{gro66}.
We content ourselves with the class of examples in Theorem \ref{27.03.17}. A complete list of cases will appear in another publication.
Recall that a stronger version of the variational Hodge conjecture (alternative Hodge conjecture in
\cite[\S 18.2]{ho13}) holds for an algebraic cycle $Z$ of codimension $\frac{n}{2}$ inside a  
smooth hypersurface of degree $d$ and dimension $n$, 
if  deformations of $Z$ as an algebraic cycle and Hodge cycle are the same.  
Let $\T$ be the open subset of $\C[x]_d$
parameterizing smooth hypersurfaces of degree $d$.  We use the notation $X_t,\ t\in\T$ and denote by $0\in\T$ the point corresponding to Fermat variety. 
We also denote by $\plc$  the trivial algebraic cycle in $X$ obtained by intersecting  a projective space $\Pn {\frac{n}{2}+1}\subset \Pn {n+1}$ 
with $X$.  For the definition of a Hodge cycle   and Hodge locus 
see \S\ref{28.03.2017}. As a corollary of Theorem \ref{14out2016} we get:   
\begin{theo}
\label{27.03.17}
Let $\check\T$ be the subvariety of $\T$ parametrizing hypersurfaces containing two linear cycle $\P^\frac{n}{2}$ and  $\check\P^\frac{n}{2}$ with 
$\P^\frac{n}{2}\cap\check\P^\frac{n}{2}=\P^m$. There is a Zariski open (and hence dense) subset $U$ of $\check\T$ such that  
the variational Hodge conjecture is true for $Z:=\P^\frac{n}{2}+\check\P^\frac{n}{2}\in X:=X_t,\ t\in U$  
with the triples $(n,d,m)$:
\begin{eqnarray*}
& &(2,d,-1),\ 5\leq d\leq 14, \\
& & (4, 4,-1), (4,5,-1), (4,6,-1), (4,5,0),  (4,6,0), \\
& & (6,3, -1), (6,4, -1), (6,4,0), \\ 
& & (8,3,-1),(8,3,0),\\
& & (10,3,-1), (10,3,0), (10,3,1), 
\end{eqnarray*}
where $\P^{-1}$ means the empty set. In particular,  if another  algebraic cycle $\check Z\subset X$  of dimension $\frac{n}{2}$ in $X$ 
satisfies $[\check Z]=b [Z]+c[\plc]$ in $H_n(X,\Q)$ for some $b,c\in\Q,\ b\not=0$, 
 then the pair $(X,\check Z)$ cannot be deformed to $(X_t,\check Z_t)$  with $t\in \T\backslash \check\T$.
\end{theo}
For larger $m$'s Theorem \ref{27.03.17} fails to be true and this is the main topic of the article \cite[Chapter 18]{ho13}.
The limitation in Theorem  \ref{27.03.17} is due to the fact that
a part of its proof  is rank computation of certain matrices, 
for which we use a computer, and  we do not know how to handle it for arbitrary $n$ and $d$. 
Theorem \ref{27.03.17} implies that  the parameter space $\check\T$ 
is an irreducible reduced component of the Hodge locus in the parameter space $\T$ of smooth hypersurfaces. 
Note that for $n=2$ the hypothesis on $\check Z$  is the same as to say that the equality  
holds in 
$\Pic(X)\otimes\Q$. By
deformation of a pair $(X,Z)$ we mean a proper flat family $g: {\cal X}\to (\C,0)$ with a closed 
subvariety ${\cal Z}\subset \cal X$ such that $g|_{\cal Z}$ is flat, 
$g^{-1}(0)=X$ and $g|_{\cal Z}^{-1}(0)=Z$. 

The Zarski open subset $U$ in Theorem \ref{27.03.17} may not contain the Fermat point as our choice
of $\P^\frac{n}{2}, \check\P^\frac{n}{2}$ for Fermat is very special, see \eqref{8mar2018} and 
\eqref{23j2019}. 
For large degree $d$, all linear cycles of dimension $\frac{n}{2}$ and inside the Fermat variety 
are of the form \eqref{22jan2019}, see \cite[\S 17.4]{ho13} and
so in order to have $0\in U$, we must verify a rank computation in \S \ref{justdance2017} 
for all possible pairs of such $\P^\frac{n}{2}, \check\P^\frac{n}{2}$.

S. Bloch in \cite{Bloch1972} proves 
variational Hodge  conjecture for semi-regular algebraic cycles which is a strong condition on 
algebraic cycles and it is not at all clear whether it  holds in  our situation.  
The only result in this direction is given in \cite{DK2016}, where the authors prove that 
any smooth projective variety $Z$ of dimension $\frac{n}{2}$ is a semi-regular sub-variety of a 
smooth projective hypersurface in  $\P^{n+1}$ of large enough degree. We can also prove similar
statements as in Theorem \ref{27.03.17} for complete intersections algebraic cycles, see \S\ref{8.03.2018}.

The strategy to prove results similar to  Theorem \ref{27.03.17} has been explained in 
the first author's book \cite[Chapters 17, 18]{ho13}.
The main tools 
are 1. the infinitesimal variation of Hodge structures (IVHS) 
developed by Carlson, Green, Griffiths and Harris  in \cite{CGGH1983} 2. A theorem of Carlson and Griffiths in 
\cite[page 7]{CarlsonGriffiths1980} which describes a Cech cohomology description
of the Griffiths' basis of the de Rham cohomology of smooth hypersurfaces, and it does not 
appear in the IVHS formulation (despite the fact that IVHS is originated from this article). 
3. the relation between  IVHS  and  the Zariski tangent space of Hodge loci as analytic spaces 4. and finally  
the computation of periods of linear cycles inside the Fermat variety,  see Theorem \ref{14out2016}. This is also the heart of our proof of  
Theorem \ref{27.03.17} which has inspired the title of the article. 
For a full exposition of old and new results on Hodge locus the reader is referred to Voisin's article \cite{voisinHL}.

\section{Infinitesimal variation of Hodge structures}
\label{28.03.2017}
Let $X\to\T$ be a family of smooth complex projective varieties, where $\T$ is irreducible and smooth. 
The main ingredient of the infinitesimal variation of Hodge structures (IVHS) at $0\in\T$ is the 
bilinear map
\begin{equation}
\label{IVHS-sept}
\TS_0\T \times   H^{\frac{n}{2}-1}(X_0,\Omega_{X_0}^{\frac{n}{2}+1}) \to   H^{\frac{n}{2}}(X_0,\Omega_{X_0}^{\frac{n}{2}})
\end{equation}
where $\TS_0\T$ is the tangent space of $\T$ at $0$. This gives us Voisin's 
$^0\bar\nabla$ map:
\begin{equation}
\label{badankoofte-1}
{^0\bar\nabla}: H^{\frac{n}{2}}(X_0,\Omega_{X_0}^{\frac{n}{2}})^{\vee}\to 
{\rm Hom} \left ( \TS_{0}\T,  
H^{\frac{n}{2}-1}(X_0,\Omega_{X_0}^{\frac{n}{2}+1})^{\vee}  \right ),
\end{equation}
where $\vee$ denotes the dual of a vector space.
A cycle $\delta_0\in H_n(X_0,\Q)$ satisfying 
$$
\mathlarger{\int}_{\delta_0}\omega=0,\ \ \forall  \omega\in F^{\frac{n}{2}+1}H^n_\dR(X_0)
$$ 
is called a Hodge cycle.  
For a Hodge cycle $\delta_0$, the integrations 
\begin{equation}
\label{21.03.2017}
\mathlarger{\int}_{\delta_0}\omega,\ \   \omega\in H^{\frac{n}{2}}(X_0,\Omega_{X_0}^{\frac{n}{2}})\cong 
\frac{F^{\frac{n}{2}}H_\dR^n(X_0)}{ F^{\frac{n}{2}+1}H^n_\dR(X_0) },\ \ \delta_0\in H_n(X_0,\Q), 
\end{equation}
are well-defined and so we get 
$\delta_0^{\podu}\in H^{\frac{n}{2}}(X_0,\Omega_{X_0}^{\frac{n}{2}})^{\vee}$. Moreover, 
$\ker( ^0\bar\nabla \delta_0^\podu )$ is the Zariski tangent space of the analytic space  $V_{\delta_0}$ with
\begin{equation}
\label{10maio16}
\O_{V_{\delta_0}}:=\O_{\T,0}\Bigg/\left\langle \mathlarger{\int}_{\delta_t}\omega_1, 
\mathlarger{\int}_{\delta_t}\omega_2, \cdots, \mathlarger{\int}_{\delta_t}\omega_a \right\rangle,
\end{equation}
at $0$,  where $\omega_1,\omega_2,\cdots,\omega_a$ are sections of 
the cohomology bundle  $H^n_\dR(X_t),\ t\in(\T,0)$ such that  
they form a basis of  $F^{\frac{n}{2}+1}H^n_\dR(X_t)$, and $\delta_t\in H_n(X_t,\Q)$ is the monodromy/parallel
transport of $\delta_0$ to $X_t$, see \cite[\S 5.3.2]{vo03}. 
The analytic space  $V_{\delta_0}$ is called the Hodge locus passing through $0$ and 
corresponding to $\delta_0$.
It might be non-reduced, see for instance \cite[Exercise 2, page 154]{vo03}.  For the full family of smooth 
hypersurfaces and $\plc$ as in Introduction, 
we have identifications
\begin{eqnarray} \label{12oct2014-1}
\TS_0\T & \cong & \C[x]_d\\ \label{12oct2014-2}
 H^{k}(X_t,\Omega_{X_t}^{n-k})   &\cong  &  (\C[x]/J)_{(k+1)d-n-2},\ \ k=0,1,\ldots,n, \ \ k\not=\frac{n}{2}
\end{eqnarray}
where $J:= \jacob(f_t)$ is the Jacobian ideal of the equation $f_t$ of $X_t$. 
For $k=\frac{n}{2}$,  $(\C[X]/J)_{(k+1)d-n-2}$ is identified with the codimension
one subspace $\ker[\plc]^\podu$ of $H^{k}(X_t,\Omega_{X_t}^{n-k})$, which is called 
the primitive part and it is in  the image of \eqref{IVHS-sept}.  
After these identifications, \eqref{IVHS-sept} is induced 
by the multiplication of polynomials. 

\section{Carlson-Griffiths theorem}
\label{6oct2017}
There can be many ways to compute hypercohomology groups. In this section in order to compute integrals \eqref{21.03.2017} we 
use a theorem of Carlson and Griffiths which gives a description of  
the algebraic de Rham cohomology of hypersurfaces using \v Cech cohomology. 
Let $X\subset\Pn {n+1}$ be a  smooth hypersurface of degree $d$ given by $f=0$. 
Recall that for a monomial $x^i=x_0^{i_0}x_1^{i_1}\cdots x_{n+1}^{i_{n+1}}$ of degree $(k+1)d-n-2$
$$
 \omega_i={\rm Residue}
 \left( 
\frac{x^i\cdot \Omega}
     {f^{k+1}}
\right)\in H^n_{\text{dR}}(X).
$$
where 
$\Omega:= \sum_{i=0}^{n+1}(-1)^ix_i\widehat{dx_i}$, $\widehat{dx_i}=dx_0\wedge dx_1\wedge\cdots dx_{i-1}\wedge dx_{i+1}\wedge\cdots \wedge dx_{n+1}$ 
and 
${\rm Residue}: H^{n+1}(\Pn {n+1}-X)\to  H^{n+2}(\Pn {n+1}, X)\cong H^n(X)$ 
is the   composition of the coboundary map with the Leray-Thom-Gysin isomorphism, 
see \cite[Chapter 4]{ho13}.
We say that $\omega_i$ has adjoint level $k$. Carlson and Griffiths in \cite{CarlsonGriffiths1980}
found an explicit expression for these forms in the algebraic de Rham cohomology of $X$ relative to 
the Jacobian covering $\mathcal{J}_X$ of $\Pn {n+1}$: 
$$
\mathcal{J}_X:=\{U_j,\ \ j=0,1,2,\cdots,n+1\},\ \ U_j:=\left\{\frac{\partial f}{\partial x_j}\neq 0\right\}.
$$
Since $X$ is smooth, this is a covering of $\Pn {n+1}$ and hence $X$ itself. 
For a vector field $Z$ in $\mathbb{C}^{n+2}$, let $\iota_Z$ denote the contraction of differential forms along $Z$ 
and for a  multi-index $j=(j_0,...,j_l)$ with $|j|:=l$ let 
\begin{eqnarray}
\Omega_j &:=& \iota_{\frac{\partial}{\partial x_l}}\circ  \iota_{\frac{\partial}{\partial x_{l-1}}}\circ \cdots \circ  \iota_{\frac{\partial}{\partial x_0}} \Omega\\
f_j &:=& \frac{\partial f}{\partial x_{j_0}}\cdot  \frac{\partial f}{\partial x_{j_1}} \cdots\frac{\partial f}{\partial x_{j_l}}.
\end{eqnarray}
\begin{theo}[Carlson-Griffiths, \cite{CarlsonGriffiths1980}, page 7]
\label{21jan2019}
 Let $\omega_i$ be a differential form of adjoint level $k$. 
 Then, in $F^{n-k}/F^{n-k+1}\cong H^k(X,\Omega_X^{n-k})$,  it is represented by the cocycle
 \begin{equation}
 \label{21dec2016}
 (\omega_i)^{n-k,k}=\frac{(-1)^{n(k+1)}}{k!}\left\{ \frac{x^i\Omega_j}{f_j}\right\}_{|j|=k}
 \end{equation} 
 with respect to the Jacobian covering.
\end{theo}
For the constant term in \eqref{21dec2016} see \cite[page 12]{CarlsonGriffiths1980}.  In order to be able to compute the integrals
of the present text explicitly and without any constant ambiguity, see Theorem \ref{14out2016}, we will need the following 
integration formula: 
\begin{equation}
 \mathlarger{\mathlarger{\int}}_{\Pn {n+1}} 
 \frac{\sum_{i=0}^{n+1}(-1)^ix_i\widehat{dx_i}}{x_0x_1\cdots x_{n+1}} =(-1)^{n+2\choose 2}(2\pi\sqrt{-1})^{n+1}. 
 \end{equation}
 The integrand  induces an element in the top algebraic de Rham cohomology $H^{2(n+1)}_\dR(\Pn {n+1})$ and   
 we have to use a canonical isomorphism between algebraic de Rham  and usual de Rham cohomology in order
 to write it as a $C^\infty$ $2(n+1)$ differential form. 
Since this will not play any role in the proof of Theorem \ref{19.03.2017} we skip its proof.

\section{Proof of Theorem \ref{14out2016}}
We prove Theorem \ref{14out2016} in the case, where $b$ is the
identity and all $a_i$'s are zero. 
In this case we simply write $\licy{\frac{n}{2}}=\licy{\frac{n}{2}}_{a,b} $.
Since $\licy{\frac{n}{2}}_{a,b}$'s are obtained by acting the automorphism group of the 
Fermat variety on a single linear cycle, the general formula easily follows.   
Let $\phi : \Pn {\frac{n}{2}}_{(y_1:\cdots:y_{\frac{n}{2}+1})}\rightarrow \Pn {n+1}_{(x_0:\cdots:x_{n+1})}$ be the immersion with the image $\licy{\frac{n}{2}}$ given by
$$
 \phi [y_1:\cdots :y_{\frac{n}{2}+1}]=[\zeta_{2d}y_1:y_1:\cdots :\zeta_{2d}y_{\frac{n}{2}+1}:y_{\frac{n}{2}+1}].
$$
We know from Carlson-Griffiths Theorem that 
\begin{equation}
(\omega_i)^{\frac{n}{2},\frac{n}{2}}= \frac{1}{\frac{n}{2}!}\left\{ \frac{x^i\Omega_j}{d^{\frac{n}{2}+1}(x_{j_0}x_{j_1}
 \cdots x_{j_{\frac{n}{2}}})^{d-1}} \right\}_{|j|=\frac{n}{2}}\in H^{\frac{n}{2}}(\mathcal{U}, \Omega^{\frac{n}{2}}),
\end{equation}
where $\mathcal{U}=\mathcal{J}_{X^d_n}$ is the standard covering of $\Pn {n+1}$ and for simplicity we have written
$\Omega^{k}= \Omega^k_{X^d_n}$. 
Therefore 
\begin{equation}
\label{2016-Bonn-1}
 \phi^*\omega_i=\frac{1}{d^{\frac{n}{2}+1}\cdot \frac{n}{2}!}
 \left\{ \frac{\zeta_{2d}^{i_0+i_2+\cdots+i_n}y^{i'}\phi^*\Omega_j}{\phi^*(x_{j_0}x_{j_1}\cdots x_{j_{\frac{n}{2}+1}})^{d-1}} \right\}_{|j|=\frac{n}{2}}\in H^{\frac{n}{2}}(\phi^{-1}(\mathcal{U}), \Omega^{\frac{n}{2}}),
\end{equation}
where $i'=(i_0+i_1,i_2+i_3,\cdots,i_n+i_{n+1})$ and $\phi^{-1}(\mathcal{U})$ is the open covering of $\Pn {\frac{n}{2}}$ given by 
the pre-images of the standard covering of 
$\Pn {n+1}$. Note that this covering has repeated open sets.  
Since for $\{k_1,k_2,\cdots, k_{\frac{n}{2}}\}\subset\{1,2,\cdots,\frac{n}{2}+1\}$ with $k_1<k_2<\cdots<k_{\frac{n}{2}}$ 
we have
$$
 (\phi^*\Omega_j)(\frac{\partial}{\partial y_{k_1}},\cdots,\frac{\partial}{\partial y_{k_{\frac{n}{2}}}})=
 $$
 $$
 \Omega(\frac{\partial}{\partial x_{j_0}},\cdots,\frac{\partial}{\partial x_{j_{\frac{n}{2}}}},\zeta_{2d}\frac{\partial}{\partial x_{2k_1-2}}+\frac{\partial}{\partial x_{2k_1-1}},\cdots,\zeta_{2d}\frac{\partial}{\partial x_{2k_{\frac{n}{2}}-2}}+\frac{\partial}{\partial x_{2k_{\frac{n}{2}}-1}}),
$$
it follows that if $\#(j\cap \{2l-2,2l-1\})=2$ for some $l\in \{1,2,\cdots,\frac{n}{2}+1\}$, then $\phi^*\Omega_j=0$. By abuse of notation here 
we have used $j$ for the set of its entries. 
On the other hand, if 
\begin{equation}
\label{10.04.2017}
\#(j\cap \{2l-2,2l-1\})=1,\ \ \ \forall l\in \{1,2,\cdots,\frac{n}{2}+1\}
\end{equation}
then 
$$
\phi^*\Omega_j(\frac{\partial}{\partial y_{k_1}},\cdots,\frac{\partial}
{\partial y_{k_{\frac{n}{2}}}})=\zeta_{2d}^{j_{\rm odd}}(-1)^{k+{\frac{n}{2} \choose 2} +j_{\rm odd}}y_k,
$$
where $k$ is the missing element, that is,  $\{k_1,\cdots,k_{\frac{n}{2}},k\}=
\{1,\cdots,\frac{n}{2}+1\}$ and $j_{\rm odd}:=\#\{0\leq i\leq \frac{n}{2},\ \ j_i \hbox{ is odd }\}$. Hence
\begin{equation}
\label{2016-Bonn-2}
 \phi^*\Omega_j=(-\zeta_{2d})^{j_{\rm odd}}(-1)^{{\frac{n}{2}\choose 2}+1}\Omega', \ \ \ \ \ \hbox{ where }\ \ \ \ \  
\Omega':=\sum_{k=1}^{\frac{n}{2}+1}(-1)^{k-1}y_k\hat{dy_k}.
\end{equation}
Since for such $j$ we have $\phi^*(x_{j_0}\cdots x_{j_{\frac{n}{2}}})^{d-1}
=\zeta_{2d}^{(d-1)(\frac{n}{2}+1-j_{\rm odd})}(y_1\cdots y_{\frac{n}{2}+1})^{d-1}$, 
replacing \eqref{2016-Bonn-2} in \eqref{2016-Bonn-1} we get 
\begin{equation}
\label{2016-Bonn-3}
\phi^*\omega_i=
\frac{(-1)^{\frac{n}{2}+1\choose 2}\zeta_{2d}^{\frac{n}{2}+1+i_0+i_2+\cdots+i_n}y^{i'}\Omega'}
{d^{\frac{n}{2}+1}\cdot\frac{n}{2}!(y_1\cdots y_{\frac{n}{2}+1})^{d-1}}
\in H^{\frac{n}{2}}(\mathcal{U}', \Omega^{\frac{n}{2}}),
\end{equation}
where $\mathcal{U}'$ is the standard covering of $\Pn {\frac{n}{2}}$. 
The form \eqref{2016-Bonn-3} is exact except for the cases in which  
$i'_{l}=d-2, \forall l\in \{1,\cdots,\frac{n}{2}+1\}$.
The result follows from the fact that the volume form $\frac{\Omega'}{y_1\cdots y_{\frac{n}{2}+1}}$ integrates $(-1)^{\frac{n}{2}+1\choose 2}(2\pi \sqrt{-1})^{\frac{n}{2}}$ over $\Pn {\frac{n}{2}}$.

\section{An elementary linear algebra problem}
\label{justdance2017}
\def\codnum{{\sf C}}
The remaining piece in the proof of Theorem \ref{27.03.17} is the following. 
For $N=d, \frac{n}{2}d-n-2$ and $(\frac{n}{2}+1)d-n-2$ let  
\begin{equation}
\label{21oct2014}
I_N:=\left \{ (i_0,i_1,\ldots,i_{n+1})\in {\mathbb Z}^{n+2}\Big| 0\leq i_e\leq d-2, \ \ i_0+i_1+\cdots+
i_{n+1}=N\right\}     
\end{equation}
We fix two  linear cycles
\begin{eqnarray}
 \label{8mar2018}
\P^\frac{n}{2}&=&\P^\frac{n}{2}_{a,b} \hbox{ with } a=(0,0,\cdots,0),\ b=(0,1,\cdots,n+1)\\ \label{23j2019}
\check \P^\frac{n}{2}&=&\P^\frac{n}{2}_{a,b} \hbox{ with } a=(\underbrace{0,0,\cdots,0}_{m+1 \hbox{ times}},
1,1,\cdots,1),\ b=(0,1,\cdots,n+1)
\end{eqnarray}
and for  $i\in I_{(\frac{n}{2}+1)d-n-2}$ we define the number 
\begin{equation}
\label{10a2017-2}
\pn_i:=\int_{\P^\frac{n}{2}}\omega_i +\int_{\check \P^\frac{n}{2}}\omega_i 
\end{equation}
where $\omega_i$ is the differential form inside the integral in  Theorem \ref{14out2016}. 
For any other \(i\) which is not in the set \(I_{(\frac{n}{2}+1)d-n-2}\), \(\pn_i\) by definition is 
zero.     Let \([\pn_{i+j}]\)  be the matrix whose rows and columns are indexed by 
\(i\in I_{\frac{n}{2}d-n-2}\) and \(j\in I_d\), 
respectively, and in its \((i,j)\) entry we have \(\pn_{i+j}\). 
For a sequence of natural numbers $\underline a=(a_1,\ldots,a_{s})$ let
us define
\begin{equation}
 \label{1julio2016-bonn}
 \codnum_{\underline{a}}=
 \bn{n+1+d}{n+1}-
 \sum_{k=1}^{s}(-1)^{k-1} \sum_{a_{i_1}+a_{i_2}+\cdots+a_{i_k}\leq d }\bn{n+1+d-a_{i_1}-a_{i_2}-\cdots-a_{i_k}}{n+1},
\end{equation}
where the second sum runs through all $k$ elements (without order) of $a_i,\ \ i=1,2,\ldots,s$. By abuse of notation we write 
$a^b:=\underbrace{a,a,\cdots,a}_{\hbox{$b$ times}}$. 
\begin{prop}
\label{21032017-2}
For the triples $(n,d,m)$ in Theorem \ref{27.03.17} we have 
\begin{equation}
\label{22.12.2016-2}
\rank([\pn_{i+j}])=2\codnum_{1^{\frac{n}{2}+1},(d-1)^{\frac{n}{2}+1}}-\codnum_{1^{n-m+1}, (d-1)^{m+1}}.
\end{equation}
\end{prop}
\begin{proof}
We verify Proposition \ref{21032017-2} by a computer. For this the reader may  download
{\tt foliation.lib}\footnote{\tt http://w3.impa.br/$\sim$hossein/foliation-allversions/foliation.lib} 
from the the first author's web page, run {\sc Singular} (see \cite{GPS01}), and
type
\begin{verbatim}
LIB "foliation.lib";
Example SumTwoLinearCycle; \end{verbatim}
Modifying $n,d,m$ arguments in the example session of the procedure {\tt SumTwoLinearCycle} one gets 
all the cases in 
Theorem \ref{27.03.17}. 
The procedures  {\tt PeriodsLinearCycle, Matrixpij, CodComIntZar} of the library {\tt foliation.lib} are
used for this verification.   
\end{proof}

\section{IVHS, periods and the proof of Theorem \ref{27.03.17}}
Let us consider the family of hypersurface $X_t$ in the usual projective space 
$\Pn {n+1}$ given by the homogeneous polynomial:
\begin{equation}
\label{15dec2016}
f_t:=x_0^{d}+x_1^{d}+\cdots+x_{n+1}^d-\sum_{j}t_j x^j=0,\ \ 
\end{equation}
$$
t=(t_j)_{j\in I}\in(\T,0), 
$$
where $x^j$ runs through $j\in I_d$. In a Zariski neighborhood of the Fermat variety, and up 
to linear transformations of $\Pn {n+1}$,   every 
hypersurface  can be written in this
format. In other words, the parametr space in \eqref{15dec2016} is transversal to the  
${\rm PGL}(n+2)$-orbits near $0$ of $\T$ in  the introduction and its projection in $\T/ {\rm PGL}(n+2)$ is etale 
near $0$.
We choose basis $x^i\in I_d$, $x^i,\ \ i\in I_{\frac{n}{2}d-n-2}, x^i, \ i\in I_{(\frac{n}{2}+1)d-n-2}$ for
$\TS_0\T$, $H^{\frac{n}{2}-1}(X_0,\Omega_{X_0}^{\frac{n}{2}+1})$ and $H^{\frac{n}{2}}(X_0,\Omega_{X_0}^{\frac{n}{2}})$, respectively.
For a Hodge cycle $\delta_0\in H_n(X_0,\Q)$, 
we write $^0\bar\nabla \delta_0^\podu$ in the above basis  and we get the matrix $[\pn_{i+j}]$, where 
$
\pn_i:=\int_{\delta_0}\omega_i
$
are the periods of $\delta_0$. This matrix has been computed for the first time in \cite{GMCD-NL}. 
For $\delta_0:=[Z_0],\ \ Z_0:=\P^{\frac{n}{2}}+ \check \P^{\frac{n}{2}}$, Theorem \ref{14out2016} gives us an explicit 
formula for the periods $\pn_i$ in \eqref{10a2017-2}.  
 Using Koszul complex one can easily see that
the right hand side of \eqref{22.12.2016-2} is the codimension of $\check \T$ in $\T$, 
see \cite[\S 17.9]{ho13}.
Knowing that $\ker [\pn_{i+j}]$ is the Zariski tangent space of the analytic space $V_{[Z_0]}$ and 
the local branch of  $(\check\T,0)$ corresponding to deformations of $Z_0$ is inside the underlying analytic 
variety of $V_{[Z_0]}$, Proposition \ref{21032017-2} implies  
that $V_{[Z_0]}$ is smooth and reduced and its underlying analytic variety  is an open subset of 
the algebraic variety $\check\T$.  Therefore, the restriction on $n$ and $d$ in our main theorem comes from
the fact that we can prove Proposition \ref{21032017-2} for the special cases of $(n,d,m)$ announced in 
Theorem \ref{19.03.2017}. 

 Since $\omega_i,\ i\in I_{\frac{n}{2}(d-2)-n-2}$ form
a basis of the primitive part of $F^{\frac{n}{2}}/F^{\frac{n}{2}+1}$ of $H^n_\dR(X)$, all the periods of 
$\plc$ are zero.
This implies that for two Hodge cycles $\delta_0,\tilde\delta_0\in H_n(X_0,\Q)$ such that  
$\delta_0=b\tilde \delta_0+c[\plc]=0$ for some 
$b,c\in\Q,\ b\not=0$, we have  $V_{\delta_0}=V_{\tilde\delta_0}$. For  
$\delta_0=[\check Z]$ and $\tilde \delta_0=[Z]$, this implies
the second part in Theorem \ref{19.03.2017}.

\section{Complete intersection algebraic cycles}
\label{8.03.2018}
Let $\C[x]_d=\C[x_0,x_1,\cdots,x_{n+1}]_d$ be the set of  homogeneous polynomials of degree $d$ in $n+2$ variables.
Assume that $n\geq 2$ is even and  $f\in \C[x]_d$ is of the following format:
\begin{equation}
\label{24.03.2017}
f=f_1f_{\frac{n}{2}+2}+f_2f_{ \frac{n}{2}+3}+\cdots+ f_{\frac{n}{2}+1}f_{n+2},\ \ f_i\in\C[x]_{d_i},\ \ f_{\frac{n}{2}+1+i}\in\C[x]_{d-d_i},
\end{equation}
where $1\leq d_i<d,\ i=1,2,\ldots, {\frac{n}{2}+1}$ is a sequence of natural numbers. 
Let $X\subset \Pn {n+1}$ be the hypersurface given by $f=0$ and $Z\subset X$ 
be the algebraic cycle given by $f_1=f_2=\cdots=f_{\frac{n}{2}+1}=0$. We call $Z$ a complete intersection algebraic cycle in $X$. 
The Fermat variety  has many of such algebraic cycles. 
Let $\T$ be the open subset of $\C[x]_d$ parameterizing smooth hypersurfaces of degree $d$ and $\T_{\underline{d}}\subset \T$ be its subset parameterizing those with \eqref{24.03.2017}.
We use the notation $X_t,\ t\in\T$ and denote by $0\in\T$ the point corresponding to Fermat variety. 
As another  corollary of Theorem \ref{14out2016} we get:   
\begin{theo}
\label{19.03.2017}
Let consider the following cases: 
\begin{enumerate}
 \item 
 $d\geq 2+\frac{4}{n}$ and $d_1=d_2=\cdots=d_{\frac{n}{2}+1}=1$, 
 \item
 $n=2, 4\leq d\leq 15$,
 \item 
 $n=4$ and $3\leq d\leq 6$,
 \item 
 $n=6$ and $3\leq d\leq 4$.
\end{enumerate}
In all these cases, except the first one,  all possible $\underline{d}$ is considered. 
There is a Zariski open (and hence dense) subset $U$ of $\T_{\underline{d}}$ such that for all $t\in U$ and  
a complete intersection algebraic cycle $Z\subset X:=X_t$ as above, 
deformations of $Z$ as an algebraic cycle and Hodge cycle are the same. 
\end{theo}
 The property in Theorem \ref{19.03.2017} is actually verified for the Fermat hypersurface 
with one of its complete intersection algebraic cycles.  
Actually, for the first case in Theorem \ref{19.03.2017} we prove that the local analytic branches of $\T_{\underline{d}}$ near 
the Fermat point are smooth and reduced. For the rest we prove this property at least for one branch.

When the first draft of this article was written, we got to know the preprint \cite[Theorem 1.1]{Dan-2014} 
in which the author states  
Theorem \ref{19.03.2017} for  arbitrary $d$. The exposition in this article can be improved, for instance 
the assumption $d>deg(Z)$ in the statement of Theorem 1.1 can be removed. The main ingredient
in this theoretical proof is Macaulay's theorem which is missing in our computational proof. We highlight that 
the advantage of our computational proof is that it works for other algebraic cycles which are not complete intersections, see Theorem \ref{27.03.17},
whereas the proof in \cite{Dan-2014} only works for complete intersections. 
The disadvantage is that one has to work with special values of $d$ and $n$ and it proves Theorem \ref{19.03.2017} 
for hypersurfaces  in a Zariski open subset of $\T_{\underline{d}}$. 
We note that  
the main result in \cite{Otwinowska2003}   implies Theorem \ref{19.03.2017} for very large degrees, however,
the lower bound in this article is not explicit and cannot be applied for a given degree.

For $n=2$ the Hodge locus is also called Noether-Lefschetz locus, and for $d_1=d_2=1$ one can even say more, 
that is namely, $\T_{1,1}$
is the only component of the Noether-Lefschetz locus with codimension $d-3$, see \cite{voisin1988, green1989}. 
For a similar statement for the case $n=2, d_1=1,d_2=2$ see \cite{voisin89}. We do not deal with this issue in this article. 
The first case in Theorem \ref{19.03.2017} is proved in \cite{GMCD-NL} and we give a new proof of this. 
The limitation in other cases is due to the fact that
a part of the proof of Theorem \ref{19.03.2017},  see Conjecture \ref{21032017} below,  
is an elementary problem in 
linear algebra, 
for which we use a computer, and apart from the first case, we do not know how to solve it in general.

The proof of Theorem \ref{19.03.2017} is similar to Theorem \ref{27.03.17}. Proposition \ref{21032017-2} is replaced with the following. Let   
\begin{equation}
\label{titular2017}
\check I_{(\frac{n}{2}+1)d-n-2}:=\left \{i\in I_{(\frac{n}{2}+1)d-n-2}\Big|i_{2l-2}+i_{2l-1}=d-2, \forall l=1,\cdots,\frac{n}{2}+1\right \}.
\end{equation}
Let also $B_1, B_2,\cdots, B_{\frac{n}{2}+1}$ be 
subsets of $\{\zeta\in \C  | \zeta^d+1=0\}$ with cardinalities $d_1,d_2,\ldots,d_{\frac{n}{2}+1}$, respectively. For 
$i\in \check I_{(\frac{n}{2}+1)d-n-2}$ we define the number 
\begin{equation}
\label{10a2017}
\pn_i:=\prod_{k=0}^{{\frac{n}{2}}}\sum_{\zeta\in B_{k+1}}\zeta^{i_{2k}+1}.
\end{equation}
For any other \(i\) which is not in the set \(\check I_{(\frac{n}{2}+1)d-n-2}\), \(\pn_i\) by definition is zero.    
\begin{conj}
\label{21032017}
We have 
\begin{equation}
\label{22.12.2016}
\rank([\pn_{i+j}])=\codnum_{d_1,d_2,\ldots,d_{\frac{n}{2}+1},d-d_1,d-d_2,\ldots,d-d_{\frac{n}{2}+1}}
\end{equation}
where the number in the right hand side is defined in \eqref{1julio2016-bonn}. 
\end{conj}
We can verify Conjecture \ref{21032017} by a computer for $n$ and $d$ given in item 2 of Theorem \ref{19.03.2017}.  
The only theoretical proof that we have is the following. 
\begin{prop}
For the case $d_1=d_2=\cdots=d_{\frac{n}{2}+1}=1$ we have 
$$
\rank[\pn_{i+j}]={\frac{n}{2}+d\choose d}-(\frac{n}{2}+1)^2.
$$
\end{prop}
\begin{proof}
Let 
$$
A:=\{i\in I_{\frac{n}{2}d-n-2}|i_0=i_2=\cdots=i_n=0\},
$$
$$
B:=\{j\in I_d|j_0=j_2=\cdots=j_n=0\}.
$$
Consider the map $\phi:B\rightarrow A$ given by $\phi(j)_{2l-2}=0$, $\phi(j)_{2l-1}=d-2-j_{2l-1}$, for $l=1,\cdots,\frac{n}{2}+1$. 
It is easy to see that $\phi$ is a bijection and  
$$
\#A=\#B={\frac{n}{2}+d\choose d}-(\frac{n}{2}+1)^2.
$$
We claim 
that the rows $\pn_{i+\bullet}, i\in A$ form a base for the image of $[\pn_{i+j}]$. Indeed, since for $(i,j)\in A\times B$
$$
\pn_{i+j} =
\left\{
	\begin{array}{ll}
		1  & \mbox{if } i=\phi(j), \\
		0 & \mbox{otherwise, } 
	\end{array}
\right.
$$
it follows that these rows are linearly independent. To see that they generate the image, it is enough to show that 
they generate all the rows. For $i\in I_{\frac{n}{2}d-n-2}$ if $i_{2l-2}+i_{2l-1}>d-2$ for 
some $l\in\{1,\cdots,\frac{n}{2}+1\}$, 
then $\pn_{i+\bullet}=0$. If not then there exists a unique $j\in B$ such that  $i+j\in \check I_{(\frac{n}{2}+1)d-n-2}$. In fact  
$j_{2l-2}=0$, $j_{2l-1}=d-2-i_{2l-2}-i_{2l-1}$, for $l=1,\cdots,\frac{n}{2}+1$. We claim that 
$$
\pn_{i+\bullet}=\zeta_{2d}^{i_0+i_2+\cdots+i_n}\cdot\pn_{\phi(j)+\bullet}.
$$
For $h\in I_d$  with $\pn_{\phi(j)+h}=0$ we have $\phi(j)+h\notin \check I_{(\frac{n}{2}+1)d-n-2}$ and  so there exists  
$l\in\{1,\cdots,\frac{n}{2}+1\}$ such that 
$$
\phi(j)_{2l-2}+\phi(j)_{2l-1}+h_{2l-2}+h_{2l-1}>d-2.
$$
Since $\phi(j)_{2l-2}+\phi(j)_{2l-1}=i_{2l-2}+i_{2l-1}$,  it follows that $\pn_{i+h}=0$. 
On the other hand, for  $h\in I_d$ with $\phi(j)+h\in \check I_{(\frac{n}{2}+1)d-n-2}$, we have  $i+h\in \check I_{(\frac{n}{2}+1)d-n-2}$ and
$$
\pn_{i+h}=\zeta_{2d}^{(i_0+h_0)+\cdots+(i_n+h_n)}=\zeta_{2d}^{i_0+\cdots+i_n}\cdot \zeta_{2d}^{h_0+\cdots+h_n}=\zeta_{2d}^{i_0+\cdots+i_n}\cdot \pn_{\phi(j)+h}.
$$
\end{proof}
Let 
\begin{equation}
Z_0: \prod_{\zeta\in B_1}(x_{0}-\zeta x_1)= \prod_{\zeta\in B_2}(x_2-\zeta x_3)=\cdots= \prod_{\zeta\in B_{\frac{n}{2}+1}}(x_{n}-
\zeta x_{n+1})=0, 
\end{equation}
where $B_i$'s are as in \S\ref{justdance2017}. 
For $\delta_0:=[Z_0]$  Theorem \ref{14out2016} implies that 
up to multiplication by a constant which does not depend on $i$ we have $\int_{Z}\omega_i=\pn_i$,  
where $\pn_i$ is defined in \eqref{10a2017}.  
 Using Koszul complex one can easily see that
the left right hand side of \eqref{22.12.2016} is the codimension of $\T_{\underline d}$ in $\T$, see \cite[Chapter 17]{ho13}.
The rest of the argument is similar to the proof of Theorem \ref{27.03.17}. Note that  the restriction on $n$ and $d$ in our main theorem comes from
the fact that we can prove Conjecture \ref{21032017} for the special cases of $n$ and $d$ announced in Theorem \ref{19.03.2017}.

\def\cprime{$'$} \def\cprime{$'$} \def\cprime{$'$} \def\cprime{$'$}


\end{document}